\documentclass{amsart}
\usepackage{amssymb, amsmath}
\usepackage{a4}
\def\qedbox{\hbox{$\rlap{$\sqcap$}\sqcup$}}

\def\gronkRMKMOREJO{
 Indeed:
\begin{enumerate}
\item For the special choice of the null vector $v=\partial_{1'}$, ${\mathcal{J}^g}(v)$ is two-step nilpotent at any point.

\item At points with $x_{3'}\neq 0$, let $v=\partial_1 -\frac{x_{1'}}{x_{3'}}\partial_3+x_2 x_{3'}\partial_{1'}$. Then ${\mathcal{J}^g}(v)$ is three-step
nilpotent. Such a degree of nilpotency is not possible  at points with $x_{3'}=0$.

\item For the special choice of the null vector $v=\partial_1 - \frac{1}{2}(x_{1'}^2-2 x_2 x_{3'})\partial_{1'}$, ${\mathcal{J}^g}(v)$ is four-step nilpotent at any point with $x_{1'} x_{3'}\neq 0$. At points with $x_{1'}=0$ take $v=\partial_1+ x_2 x_{3'}\partial_{1'}+\partial_{3'}$, and at points with $x_{3'}=0$ take $v= \partial_1-\frac{1}{2}x_{1'}^2 \partial_{1'}+\partial_{3'}$, to get a four-step nilpotent null Jacobi operator.
\end{enumerate}}

\def\gronkSZABO{
To be more explicit. Given a non-null vector $v=\sum_{i=1}^3(\alpha_i\partial_i+\alpha_{i'}\partial_{i'})$,
the associated Szab\'{o} operator, when expressed in the coordinate basis, takes the form
\[
\nabla_v{\mathcal{J}^g}_v=\left(\begin{array}{cc}
A & 0\\
B & {}^tA
\end{array}\right),
\]
where
\[
    A=\left(\begin{array}{ccc}
    2x_{3'} \alpha_1^2\alpha_2 & -2x_{3'}\alpha_1^3 & 0\\
    2x_{3'} \alpha_1\alpha_2^2 & -2x_{3'}\alpha_1^2 \alpha_2 & 0\\
    -2\alpha_1\alpha_2(x_{1'}\alpha_1+x_{2'}\alpha_2)
        & 2\alpha_1^2(x_{1'}\alpha_1+x_{2'}\alpha_2) & 0
\end{array}\right).
\]
Hence the characteristic polynomial of the Szab\'{o} operators is
$$p_\lambda(\nabla_v{\mathcal{J}^g}_v)=\lambda^6$$ (independently of the $3\times 3$-matrix $B$), so metric  is Szab\'{o} of signature $(3,3)$
with zero eigenvalues. A further analysis shows that $\nabla_v{\mathcal{J}^g}_v$ can be one-step nilpotent, or two-step nilpotent, or three-step nilpotent, changing
even at a fixed point, and therefore the metric is not pointwise Jordan Szab\'{o}. Indeed:
\begin{enumerate}
\item For the special choice of $v=\partial_3 + \frac{1}{2}(\varepsilon-x_{3'}^2) \partial_{3'}$,
 we have $g(v,v)=\varepsilon$, and  $\nabla_v{\mathcal{J}^g}_v$ is one-step nilpotent at any point.
\item For the special choice of
$v=\partial_2 + \frac{1}{2}(\varepsilon-x_{2'}^2)\partial_{2'}+\lambda \partial_{3'}$, we have
$g(v,v)=\varepsilon$, and  choosing   appropriate $\lambda\in\mathbb{R}$  we obtain that
$\nabla_v{\mathcal{J}^g}_v$ is two-step nilpotent at any point.
\item For the special choice of
$v=\partial_{1}+\frac{1}{2}(\varepsilon-(x_{1'}^2-2 x_2 x_{3'}))\partial_{1'}+ \lambda \partial_{3'}$, we
have $g(v,v)=\varepsilon$, and  choosing appropriate $\lambda\in\mathbb{R}$  we obtain that
$\nabla_v{\mathcal{J}^g}_v$ is three-step nilpotent at any point.
\end{enumerate}}

\def\gronkJORDONBLOCK{
 More precisely,
if
$v=\sum_{i=1}^3(\alpha_i\partial_i+\alpha_{i'}\partial_{i'})$, setting
\[
\begin{array}{l}
    \xi_1 = x_{3'}^2 (\varepsilon_v
    + x_2 x_{3'} \alpha_1^2 - \alpha_1 \alpha_{1'} - \alpha_2 \alpha_{2'} + \alpha_3 \alpha_{3'})
    \\[0.035in]
    \phantom{\xi_1=}
    +  2 x_{3'} ( x_{1'} \alpha_1 +  x_{2'} \alpha_2 ) \alpha_{3'}
    + 2 \alpha_{3'}^2,
    \\[0.035in]
    \xi_2 = \alpha_2 \left(   x_{2'}^2 \alpha_2^2 \alpha_{3'}
    + x_{3'}^2 \alpha_3 (\varepsilon_v + \alpha_3 \alpha_{3'})
    +    2 \alpha_{3'} (\alpha_2 \alpha_{2'} + \alpha_3 \alpha_{3'})
    \right.
    \\[0.035in]
    \phantom{\xi_2 = \alpha_2 (}
    \left.
    +     x_{2'} x_{3'} \alpha_2 (\varepsilon_v + 2 \alpha_3 \alpha_{3'})
    \right),
\end{array}
\]
we have:
\begin{enumerate}
\item If $\alpha_1$ and $\xi_1$ do not vanish, then ${\mathcal{J}^g}(v)$ has a $3\times 3$ Jordan block.

\item If $\alpha_1=0$ and $\xi_2\neq 0$, or otherwise $\xi_1=0$ and $\alpha_1\neq 0$, then ${\mathcal{J}^g}(v)$  has a single $2\times 2$
Jordan block or has two $2\times
2$ Jordan blocks.

\item If $\alpha_1=0$ and $\xi_2=0$,  then ${\mathcal{J}^g}(v)$ is diagonalizable.
\end{enumerate}
}

\def\gronkJNF
{
\begin{proof}  Set
\[
    \mathcal{A}^r(v) = {\mathcal{J}^g}(v)\cdot
    ({\mathcal{J}^g}(v)- \varepsilon_v \operatorname{id})\cdot
    \left({\mathcal{J}^g}(v)-\textstyle\frac{\varepsilon_v}{4}\operatorname{id}\right)^r,\quad
    r=1,\dots,4,
\]
for any unit vector $v$. First, a very long but  direct computation shows that \begin{equation}\label{dim 6 - minimo 1}
    \mathcal{A}(v) =
    \left(
    \begin{array}{cc}
        A & 0
        \\
        B & A^{t}
    \end{array}
    \right),
\end{equation}
where
\[
    A=\frac{3}{16}\varepsilon_v \left(
    \begin{array}{ccc}
        - \eta \alpha_1^2 \alpha_2 & \eta \alpha_1^3 & 0
        \\
        -\eta \alpha_1 \alpha_2^2 & \eta \alpha_1^2 \alpha_2 & 0
        \\
        \alpha_1 \alpha_2 (\varepsilon_v - \eta \alpha_3) &
        -\alpha_1^2 (\varepsilon_v - \eta \alpha_3) & 0
    \end{array}
    \right),
\]
with $\eta = x_{3'} (x_{1'} \alpha_1 + x_{2'} \alpha_2 + x_{3'} \alpha_3) + 2 \alpha_{3'}$, and $B=\frac{\varepsilon_v}{16}(b_{ij})$ is given by
\begin{flushleft}
$
\begin{array}{l}
    b_{11}=     2\alpha_2
    \left\{3 x_{3'} \alpha_2 (x_{2'} \alpha_1 \alpha_{1'}
                                - x_{1'} \alpha_1 \alpha_{2'} + 2 x_{2'} \alpha_3 \alpha_{3'})
    \right.
    \\[0.035in]
    \phantom{b_{11}=}
    + x_{3'}^2 ((2 - 3 x_2 x_{2'}) \alpha_1^2 \alpha_2 +
    3 \alpha_3 (\varepsilon_v - \alpha_2 \alpha_{2'} + \alpha_3 \alpha_{3'}))
    \\[0.035in]
    \phantom{b_{11}=}
    \left.
    - 3 (4 \varepsilon_v \xi_1 \alpha_1^2 \alpha_2 + x_{1'} x_{2'} \alpha_1 \alpha_2 \alpha_{3'}
    + \alpha_3 (\xi_1 - 2 \alpha_{3'}^2))
    \right\}
    -6 \xi_2 ,
\end{array}
$
\end{flushleft}
\begin{flushleft}
$
\begin{array}{l}
       b_{12}= \alpha_1 \left\{
       6 \xi_1 (4 \varepsilon_v \alpha_1^2 \alpha_2 + \alpha_3) +
   3 x_{1'} x_{3'} \alpha_1 (\alpha_1 \alpha_{1'} + 3 \alpha_2 \alpha_{2'})
   \right.
   \\[0.035in]
    \phantom{b_{12}=}
    -    x_{3'}^2 (4 \alpha_1^2 \alpha_2 - 3 \alpha_1 \alpha_3 \alpha_{1'} +
      3 \alpha_3 (\varepsilon_v - 3 \alpha_2 \alpha_{2'})) + 6 x_{2'}^2 \alpha_2^2 \alpha_{3'}
   \\[0.035in]
    \phantom{b_{12}=}
    +
   3 x_{2'} \alpha_2 (x_{3'} (2 \varepsilon_v + 2 x_2 x_{3'} \alpha_1^2 - \alpha_1 \alpha_{1'} +
         \alpha_2 \alpha_{2'})
    + 2 x_{1'} \alpha_1 \alpha_{3'})
    \\[0.035in]
    \phantom{b_{12}=}
    \left.
    -
   3 (\varepsilon_v - 2 \alpha_1 \alpha_{1'} - 6 \alpha_2 \alpha_{2'}) \alpha_{3'}
         \right\},
\end{array}
$
\end{flushleft}
\begin{flushleft}
$
\begin{array}{l}
       b_{13}= 3 \alpha_1 \alpha_2 \left\{
       \varepsilon_v x_{3'}^2  + x_{3'} (x_{1'} \alpha_1 + x_{2'} \alpha_2 + x_{3'} \alpha_3) \alpha_{3'} +
   2 \alpha_{3'}^2
   \right\},
\end{array}
$
\end{flushleft}
\begin{flushleft}
$
\begin{array}{l}
       b_{21}= \alpha_1 \left\{
       24 \varepsilon_v \xi_1 \alpha_1^2 \alpha_2
       -    3 x_{1'} x_{3'} \alpha_1 (\alpha_1 \alpha_{1'} - \alpha_2 \alpha_{2'})
   \right.
   \\[0.035in]
    \phantom{b_{21}=}
   - x_{3'}^2 (4 \alpha_1^2 \alpha_2 + 3 \alpha_1 \alpha_3 \alpha_{1'} -
      3 \alpha_3 (\varepsilon_v + \alpha_2 \alpha_{2'})) + 6 x_{2'}^2 \alpha_2^2 \alpha_{3'}
    \\[0.035in]
    \phantom{b_{21}=}
    +
   3 x_{2'} \alpha_2 (x_{3'} (2 \varepsilon_v + 2 x_2 x_{3'} \alpha_1^2 - 3 \alpha_1 \alpha_{1'} -
         \alpha_2 \alpha_{2'})
    + 2 x_{1'} \alpha_1 \alpha_{3'})
    \\[0.035in]
    \phantom{b_{21}=}
    \left.
    +
   3 (\varepsilon_v - 2 \alpha_1 \alpha_{1'} + 2 \alpha_2 \alpha_{2'}) \alpha_{3'}

         \right\},
\end{array}
$
\end{flushleft}
\begin{flushleft}
$
\begin{array}{l}
       b_{22}= -2 \alpha_1^2 \left\{
       3 x_{1'} x_{3'} \alpha_1 \alpha_{2'} - x_{3'}^2 (2 \alpha_1^2 - 3 \alpha_3 \alpha_{2'}) +
   3 x_{2'}^2 \alpha_2 \alpha_{3'}
   \right.
   \\[0.035in]
    \phantom{b_{22}=}
    \left.
    +
   3 x_{2'} (x_{3'} (\varepsilon_v + \alpha_1 (x_2 x_{3'} \alpha_1 - \alpha_{1'})) +
      x_{1'} \alpha_1 \alpha_{3'}) + 6 (2 \varepsilon_v \xi_1 \alpha_1^2 + \alpha_{2'} \alpha_{3'})
\right\},
\end{array}
$
\end{flushleft}
\begin{flushleft}
$
\begin{array}{l}
       b_{23}= -3 \alpha_1^2 \left\{
       \varepsilon_v x_{3'}^2  +
   x_{3'} (x_{1'} \alpha_1 + x_{2'} \alpha_2 + x_{3'} \alpha_3) \alpha_{3'} + 2 \alpha_{3'}^2
   \right\},
\end{array}
$
\end{flushleft}
\begin{flushleft}
$
\begin{array}{l}
       b_{31}= -3 \alpha_1 \alpha_2 \left\{
       \varepsilon_v x_{3'}^2  - 2 \xi_1 +
   x_{3'} (x_{1'} \alpha_1 + x_{2'} \alpha_2 + x_{3'} \alpha_3) \alpha_{3'} + 2 \alpha_{3'}^2
   \right\},
\end{array}
$
\end{flushleft}
\begin{flushleft}
$
\begin{array}{l}
       b_{32}= 3 \alpha_1^2 \left\{
       \varepsilon_v x_{3'}^2  - 2 \xi_1 +
   x_{3'} (x_{1'} \alpha_1 + x_{2'} \alpha_2 + x_{3'} \alpha_3) \alpha_{3'} + 2 \alpha_{3'}^2
   \right\},
\end{array}
$
\end{flushleft}
\begin{flushleft}
$
\begin{array}{l}
       b_{33}= 0.
\end{array}
$
\end{flushleft}
Moreover,
\begin{equation}\label{dim 6 - minimo 2}
    \mathcal{A}^2(v) = \frac{3}{8} \varepsilon_v \,\xi_1
    \left(
    \begin{array}{cccccc}
        0 & 0 & 0 & 0 & 0 & 0
        \\
        0 & 0 & 0 & 0 & 0 & 0
        \\
        0 & 0 & 0 & 0 & 0 & 0
        \\
        -\alpha_1^2 \alpha_2^2 & \alpha_1^3 \alpha_2 & 0 & 0 & 0 & 0
        \\
        \alpha_1^3 \alpha_2 & - \alpha_1^4 & 0 & 0 & 0 & 0
        \\
        0 & 0 & 0 & 0 & 0 & 0
    \end{array}
    \right).
\end{equation}
It follows that $\mathcal{A}^3(v)=0$, and therefore the Jacobi operator ${\mathcal{J}^g}(v)$ associated to a unit vector $v$ is either diagonalizable, has a single $2\times 2$
Jordan block, or has a $3\times 3$ Jordan block, or has two $2\times
2$ Jordan blocks. 

Note that $(a)$ follows from (\ref{dim 6 - minimo 2}). Now, assuming  $\mathcal{A}^2(v)=0$, i.e., $\alpha_1=0$ or $\xi_1=0$, we analyze the vanishing of $\mathcal{A}(v)$ using (\ref{dim 6 - minimo 1}).  If $\alpha_1=0$, then the only non-vanishing element of  $\mathcal{A}(v)$ is $(\mathcal{A}(v))_{1'1}=-\frac{3}{8} \varepsilon_v \xi_2$; if $\xi_1=0$, then
$(\mathcal{A}(v))_{12}=\frac{3}{16} \varepsilon_v \gamma \alpha_1^3$ and $(\mathcal{A}(v))_{32}=-\frac{3}{16} \varepsilon_v \alpha_1^2 (\varepsilon_v - \gamma \alpha_3)$, for  $\gamma=x_{3'}( x_{1'}\alpha_1+x_{2'} \alpha_2 +x_{3'} \alpha_3)+2 \alpha_{3'}$, which implies that $\mathcal{A}(v)$ does not diagonalize if $\alpha_1\neq 0$. Thus one easily gets $(b)$ and $(c)$.
\end{proof}
}

\def\gronkJORDONREMARK{
Note that $(a)$, $(b)$ and $(c)$  are possible at any point, and for spacelike or timelike unit
vectors. Indeed, the following can be easily checked:
\begin{itemize}
\item[(i)] For the special choice of  $v=\partial_1 - \frac{1}{2}(x_{1'}^2-2 x_2 x_{3'}- \varepsilon)\partial_{1'}+\lambda \partial_{3'}$, we have $g(v,v)=\varepsilon$,   $\alpha_1=1\neq 0$, and $\xi_1=\frac{1}{2} (4 \lambda^2 + 4 \lambda x_{1'} x_{3'} + x_{3'}^2 (x_{1'}^2 + \varepsilon))$. Note that, for a fixed point, $\lambda$ can be chosen so that $\xi_1\neq 0$. So $(a)$ holds.

\item[(ii)] For the special choice of $v=\partial_2 -\frac{1}{2}(x_{2'}^2-\varepsilon)\partial_{2'}+\lambda\partial_{3'}$, we have $g(v,v)=\varepsilon$, $\alpha_1=0$, and $\xi_2=(\lambda+x_{2'} x_{3'})\varepsilon$. Note that, for a fixed point,  $\lambda$ can be chosen so that $\xi_2\neq 0$. So $(b)$ holds.

\item[(iii)] For the special choice of  $v=\partial_3 + \frac{1}{2}(\varepsilon -x_{3'}^2)\partial_{3'}$, we have $g(v,v)=\varepsilon$, and $\alpha_1=\xi_2=0$. So $(c)$ holds.

\end{itemize}
\medskip
}

\newtheorem{theorem}{Theorem}[section]
\newtheorem{lemma}[theorem]{Lemma}

\newtheorem{remark}[theorem]{Remark}

\makeatletter
  
  \@addtoreset{equation}{section}
 \makeatother

\begin{document}

\title
{The geometry of modified Riemannian extensions}
\author{E. Calvi\~{n}o-Louzao, E. Garc\'{\i}a-R\'{\i}o, P. Gilkey, and  R. V\'{a}zquez-Lorenzo}
\address{C-L, G-R, V-L: Department of Geometry and Topology, Faculty of Mathematics,
University of Santiago de Compostela, 15782 Santiago de Compostela,
Spain}
\email{estebcl@edu.xunta.es, Eduardo.Garcia.Rio@usc.es, ravazlor@edu.xunta.es}
\address{G: Mathematics Department, University of Oregon, Eugene, Oregon 97403, USA }
\email{gilkey@uoregon.edu}
\thanks{Supported by projects MTM2006-01432
and PGIDIT06PXIB207054PR (Spain).}
\begin{abstract}
We show that every paracomplex space form is locally isometric to a modified Riemannian extension and give necessary and sufficient conditions so that
a modified Riemannian extension is Einstein. We exhibit Riemannian extension Osserman manifolds of signature $(3,3)$ whose Jacobi operators have
non-trivial Jordan normal form and which are not nilpotent. We present new four dimensional results in Osserman
geometry.
\end{abstract}
\maketitle

\section{Introduction}
Walker metrics are indecomposable pseudo-Riemannian metrics
which are not irreducible (i.e., they admit a null parallel
distribution). They play a distinguished role in geometry and physics
\cite{ACGL, BBI, DR, DR-GR-VL-06, Honda-Tsukada, LM, Mag1982}. Lorentzian Walker
metrics have been studied extensively in the physics literature
since they constitute the background metric of the pp-wave models
(see for example \cite{FS} and references therein). From a purely
geometric point of view, Lorentzian Walker metrics naturally appear
in the investigation  of the non-uniqueness of the metric for the Levi Civita
connection \cite{MT}, a question which is meaningless in the
positive definite case. Moreover, Walker metrics are the underlying
structure of many geometrical structures like para-Kaehler and
hypersymplectic structures (see \cite{BK, CMMS, CFG96, Hit1990, IZ2005}).

The simplest examples of non Lorentzian Walker metrics are provided by
the socalled Riemannian extensions. This construction, which relates
affine and pseudo-Riemannian geometries, associates a neutral
signature metric on $T^*M$ to any torsion free connection $\nabla$
on the base manifold $M$. Riemannian extensions have been used both to
understand questions in affine geometry and to solve curvature
problems (see for example \cite{Af, BCGHV, GKVV}). It
is a remarkable fact that Walker metrics satisfying some natural
curvature conditions are locally Riemannian extensions, thus leading
the corresponding classification problem to a task in affine
geometry as shown in \cite{C-GR-VL, D, DR}.

In this paper we introduce a modification of the usual Riemannian
extensions with special attention to the behaviour of their
curvature. The geometry of modified Riemannian extensions is much less
rigid than that of the Riemannian extensions, allowing the existence of
many non Ricci flat Einstein metrics, which can be further
specialized to be Osserman (i.e., the eigenvalues of the Jacobi
operators are constant on the unit pseudo-sphere bundles) since
their scalar curvature invariants do not vanish. In particular, we
show that any paracomplex space form is locally a modified Riemannian
extension where the corresponding torsion free connection is
necessarily flat (cf. Theorem \ref{thm-1.1}). This description seems
to be well suited to further investigations, specially for
the consideration of the Lagrangian submanifolds of paracomplex
space forms.

Modified Riemannian extensions turn out to be very useful in describing
four dimensional Walker geometry. Indeed, we show in Theorem
\ref{thm-7.2} that any self-dual Walker metric is a modified Riemannian
extension. As a consequence, a description of all four dimensional
Osserman metrics whose Jacobi operator has a non-zero double root of
its minimal polynomial is given in Theorem \ref{thm-7.3}. This
result, coupled with recent work of Derdzinski \cite{De}, completes
the classification of four dimensional Osserman metrics. Finally, as
an application of the four dimensional results, one obtains a
procedure to construct new Osserman metrics in higher dimensions, a
problem that has been a task in the field (cf. Theorem
\ref{thm-1.3}).

\section{Summary of results}

\subsection{Affine geometry} Let $R^\nabla$ be the curvature operator of a torsion free connection $\nabla$ on the
tangent bundle of a smooth manifold $M$ of dimension $n$; if $X$ and $Y$ are smooth vector fields on $M$, then
$$R^\nabla(X,Y):=\nabla_X\nabla_Y-\nabla_Y\nabla_X-\nabla_{[X,Y]}\,.$$
The {\it Ricci tensor} $\rho^\nabla$ is defined by contracting indices:
$$\rho^\nabla(X,Y):=\operatorname{Tr}\{Z\rightarrow R^\nabla(Z,X)Y\}\,.$$
In contrast to the Riemannian setting, $\rho^\nabla$ need not be a symmetric $2$-tensor field. We refer to \cite{BO} for
further details concerning curvature decompositions. We denote the symmetric and anti-symmetric Ricci tensors by:
$$
\begin{array}{l}
    \rho^{\nabla,s}(X,Y):=\textstyle\frac12\{\rho^\nabla(X,Y)+\rho^\nabla(Y,X)\}\,,
    \\[0.05in]
  \rho^{\nabla,a}(X,Y):=\textstyle\frac12\{\rho^\nabla(X,Y)-\rho^\nabla(Y,X)\}\,.$$
\end{array}
$$
\subsection{The modified Riemannian extension}  Let
$\Phi\in C^\infty(S^2(T^*M))$ be a symmetric $2$-tensor field and let
$T,S\in C^\infty(\operatorname{End}(TM))$ be $(1,1)$-tensor fields on $M$. In Section
\ref{sect-4} we will use these data to define (see Equation (\ref{eqn-4.b})) a neutral signature
pseudo-Riemannian metric
$g_{\nabla,\Phi,T,S}$ on the cotangent bundle $T^*M$ which is
called the {\it modified Riemannian extension} which is a {\it Walker metric}.  The case
$T=c\operatorname{id}$ and $S=\operatorname{id}$ is of particular importance in our treatment
and will be denoted by $g_{\nabla,c}$ if $\Phi=0$ and by $g_{\nabla,\Phi,c}$ if $\Phi\ne0$.

\subsection{Einstein geometry} We have the following result
\begin{theorem}\label{thm-1.2}
The modified Riemannian extension $g_{\nabla,\Phi,c}$ on the cotangent bundle of an $n$ dimensional affine manifold is Einstein if and only
if
$\Phi=\frac{4}{c(n-1)}\rho^{\nabla,s}$.
\end{theorem}

\subsection{Para-Kaehler geometry} Our fundamental result
concerning such geometry is the following which illustrates the importance of the modified Riemannian extension:
\begin{theorem}\label{thm-1.1}
A para-Kaehler metric of non-zero constant
para holomorphic sectional curvature $c$ is locally isometric to the
cotangent bundle of an affine manifold which is
equipped with the modified Riemannian extension $g_{\nabla,c}$ where
$\nabla$ is a flat connection.
\end{theorem}

\subsection{Osserman geometry} We can use the modified Riemannian extension to exhibit the following examples which extend previous results in signature $(2,2)$
(see Theorem 2.1 of
\cite{DGV1}) -- it has long been a task in this field to build examples of Osserman manifolds which were not nilpotent and which exhibited non-trivial
Jordan normal form.

\begin{theorem}\label{thm-1.3}
Let $M=\mathbb{R}^3$ and let $\nabla$ be the torsion free connection whose only non-zero
Christoffel symbol is $\nabla_{\partial_{1}}\partial_{1}=x_2\partial_{3}$. Let $g:=g_{\nabla,1}$ on $T^*M$. Then $g$ is an Osserman metric of
signature $(3,3)$ with eigenvalues $\{0,1,\frac14,\frac14,\frac14,\frac14\}$ which is neither spacelike Jordan Osserman nor timelike Jordan Osserman
at any point. The Jacobi operator is neither diagonalizable nor nilpotent for a generic tangent vector.
\end{theorem}

\subsection{Notational conventions} We shall let $\mathcal{M}:=(M,\nabla)$ denote an affine manifold and $R$ the
associated curvature operator. Similarly, we shall let $\mathcal{N}:=(N,g)$ denote a pseudo-Riemannian manifold,
$\nabla^g$ denote the associated Levi-Civita connection, and $R^g$ the associated curvature operator. If $\xi_i\in T^*N$, then the symmetric
product is defined by
$$\xi_1\circ\xi_2:=\textstyle\frac12(\xi_1\otimes\xi_2+\xi_2\otimes\xi_1)\in
S^2(T^*N).$$

\subsection{Outline of the paper} In Section \ref{sect-2}, we establish notation and recall some basic definitions in the geometry of the curvature
operator. We shall also discuss Osserman, Szab\'o, and Ivanov--Petrova geometry. In Section
\ref{sect-3} we present a brief introduction to Walker geometry and
in Section \ref{sect-4} we define the modified Riemannian extension and establish Theorem
 \ref{thm-1.2}. In Section
\ref{sect-5} we discuss para-Kaehler geometry and prove Theorem
\ref{thm-1.1}. In Section \ref{sect-7} we present some results in
four dimensional geometry with a description of four dimensional Osserman metrics
whose Jacobi operators have a non-zero double root of its minimal polynomial, as a
generalization of paracomplex space forms (cf. Theorem \ref{thm-7.3}).
We conclude in Section \ref{sect-6} by proving Theorem \ref{thm-1.3}
and discussing some additional results for this six dimensional example that relate to Szab\'o and
Ivanov--Petrova geometry.
Throughout, we shall adopt the {\it
Einstein convention} and sum over repeated indices. We shall
suppress many of the technical details in the interests of brevity
in giving various proofs in this paper -- further details are
available from the authors upon request.

\section{The geometry of the curvature operator}\label{sect-2}

\subsection{Osserman geometry}  Let
$\mathcal{J}^g(X):Y\rightarrow R^g(Y,X)X$ be the Jacobi operator on a pseudo-Riemannian manifold $\mathcal{N}$ of
signature $(p,q)$. One says that $\mathcal{N}$ is {\it timelike Osserman} (resp. {\it spacelike Osserman}) if the eigenvalues of $\mathcal{J}^g$ are
constant on the pseudo-sphere of unit timelike (resp. spacelike) vectors; these are equivalent concepts if $p>0$ and $q>0$ \cite{BBG,GKVa,GKVV}.
Similarly, we say that
$\mathcal{N}$ is {\it timelike Jordan Osserman} or {\it spacelike Jordan Osserman} if the Jordan normal form of $\mathcal{J}^g$ is
constant  on the appropriate pseudo-sphere; these are in general not equivalent concepts.

\subsection{Szab\'o geometry} A pseudo-Riemannian manifold $\mathcal{N}$ is said to be \emph{Szab\'{o}} if the {\it Szab\'o operator}
$\mathcal{S}(X):Y\rightarrow\nabla_XR^g(Y,X)X$ has constant eigenvalues on
$S^{\pm}(TN)$
\cite{G-I-S}. Any Szab\'{o} manifold is locally symmetric in the Riemannian \cite{Szabo} and the Lorentzian  \cite{G-S} setting but
the higher signature case supports examples with nilpotent Szab\'{o} operators (cf. \cite{G-I-S} and the references therein).

\subsection{Ivanov--Petrova geometry}
For any oriented non-degenerate $2$-plane
$\pi$, the \emph{skew-symmetric curvature operator} $\mathcal{R}^g(\pi)$ of the Levi-Civita connection is defined by
\[
   \mathcal{R}^g(\pi)=\left|g(X,X) g(Y,Y) - g(X,Y)^2\right|^{-1/2}R^g(X,Y)\,;
\]
$\mathcal{R}^g(\pi)$ is a skew-adjoint operator which is independent of the oriented
basis $\{ X,Y\}$ of $\pi$. $\mathcal{N}$ is said to be spacelike (respectively, timelike or mixed)
\emph{Ivanov--Petrova} if the eigenvalues of
$\mathcal{R}^g(\pi)$ are constant on the appropriate Grassmannian (see \cite{IP2-98,Z-00}). If $p\ge2$ and $q\ge2$, these are equivalent
conditions
\cite{G1-01} so one simply says the metric is Ivanov--Petrova in this setting.

\section{Walker geometry}\label{sect-3}
A \emph{Walker manifold} is a triple $(N,g,\mathcal{D})$, where $N$
is an $n$ dimensional manifold, where $g$ is a pseudo-Riemannian metric of signature $(p,q)$ on $N$, and where
$\mathcal{D}$ is an $r$ dimensional parallel null distribution.

\subsection{Geometrical contexts} Walker
metrics appear as the underlying structure of several specific
pseudo-Riemann\-ian structures. For instance, indecomposable metrics
which are not irreducible play a distinguished role in investigating
the holonomy of indefinite metrics. Those metrics are naturally
equipped with a Walker structure (see for example \cite{BBI} and the
references therein). Einstein hypersurfaces in indefinite real space
forms with two-step nilpotent shape operators \cite{Mag1982} are
Walker.  Similarly, locally conformally flat manifolds with nilpotent
Ricci operator are Walker manifolds \cite{Honda-Tsukada}. Also, non-trivial conformally
symmetric manifolds (i.e., neither symmetric nor locally conformally
flat) may only occur in the pseudo-Riemannian setting and they are
Walker manifolds \cite{DR}.

\subsection{Neutral signature Walker manifolds} Of special interest are those manifolds admitting a field of null
planes of maximum dimension $r=\frac{n}{2}$. This is the case of
para-Kaehler \cite{IZ2005} and hyper-symplectic structures
\cite{Hit1990}. Note that, in opposition to the non-degenerate case,
the complementary distribution to a parallel degenerate plane field
is not necessarily parallel (even not integrable). Moreover,
parallelizability of the complementary plane field is indeed
equivalent to the existence of a para-Kaehler structure \cite{BBI}.
Note that any four dimensional Osserman  manifold of neutral signature whose
Jacobi operators have a non-zero double root of their minimal
polynomial is necessarily  Walker with a parallel field of planes of
maximal dimensionality \cite{BBR01,DR-GR-VL-06}.

\subsection{Walker coordinates}
 Walker \cite{W-50} (see also the discussion in \cite{DR1}) constructed simplified local coordinates in this setting. We
shall restrict our attention to neutral signature $(p,p)$. Let $(N^n,g,\mathcal{D})$ be a Walker manifold of signature $(p,p)$ where
$\dim\{\mathcal{D}\}=p$. There are local
coordinates $(x_1, \dots, x_p, x_{1'},\dots,
x_{p'})$ so that
\begin{equation}\label{eqn-2.a}
g= 2\, dx^i\circ dx^{i'}+B_{ij}dx^i\circ dx^j\,.
\end{equation}
Here $B$ is a symmetric matrix and
the parallel degenerate distribution is given by
$$\mathcal{D}=\operatorname{Span}\{\partial_{1^\prime},...,\partial_{p^\prime}\}
\quad\text{where}\quad\partial_i:=\textstyle\frac{\partial}{\partial x_i}\quad\text{and}\quad\partial_{i'}:=\frac\partial{\partial x_{i'}}\,.$$

\subsection{ The Christoffel symbols of $\nabla^g$} We sum over $1\le s\le p$:
\medbreak\qquad
$\Gamma^g_{ij}{}^{k}=-\frac{1}{2} \partial_{k'} g_{ij}$,\qquad\qquad
$\Gamma^g_{i'j}{}^{k'}=\frac{1}{2} \partial_{i'} g_{jk}$,
\medbreak\qquad
$\Gamma^g_{ij}{}^{k'}=\frac{1}{2}
        \left(
        -\partial_{k} g_{ij} + \partial_{j} g_{ik} + \partial_{i} g_{jk}
        +g_{ks}  \partial_{s'} g_{ij}
        \right)$.

\subsection{ The Riemann curvature tensor of $\nabla^g$}\label{Sect-TRC}
 We sum over $1\le s\le p$, $1\le t\le p$:
\medbreak\qquad
$R^g_{ijk}{}^{h} = -\frac{1}{2}\left(
        \partial_{i}\partial_{h'} g_{jk}
        -\partial_{j}\partial_{h'} g_{ik}
        \right)
        -\frac{1}{4}
        \left(\partial_{s'} g_{ik} \partial_{h'} g_{js}
        -\partial_{s'}g_{jk} \partial_{h'} g_{is}
        \right)$,
\medbreak\qquad
$R^g_{ijk}{}^{h'}  =
        -\frac{1}{2}\left(
                \partial_j\partial_k g_{ih} - \partial_j\partial_h g_{ik}
                + \partial_i\partial_h g_{jk} - \partial_i\partial_k
                g_{jh}\right)$
\smallbreak\qquad\qquad $-\frac{1}{4}\left\{\partial_{s'} g_{ik}\left(
        \partial_{h} g_{js} - \partial_{s} g_{jh} - \partial_{j} g_{sh}-g_{ht}  \partial_{t'} g_{js}
        \right)\right.$
\smallbreak\qquad\qquad $- \partial_{s'} g_{jk}
        \left(
        \partial_{h} g_{is} - \partial_{s} g_{ih} - \partial_{i} g_{sh}
        -g_{ht}  \partial_{t'} g_{is}
        \right)$
\smallbreak\qquad\qquad$-\partial_{s'} g_{jh}\left(
        \partial_{s} g_{ik} - \partial_{k} g_{is} - \partial_{i} g_{ks}
        -g_{st}  \partial_{t'} g_{ik}
        \right)$
\smallbreak\qquad\qquad $+\partial_{s'} g_{ih}
        \left(
        \partial_{s} g_{jk} - \partial_{k} g_{js} - \partial_{j} g_{ks}
        - g_{st}  \partial_{t'} g_{jk}
        \right)$
\smallbreak\qquad\qquad $\left.
        +2 \partial_j \left( g_{hs} \partial_{s'}g_{ik}\right)
        -2 \partial_i \left( g_{hs} \partial_{s'}g_{jk}\right)
        \right\}$,
\medbreak\qquad
$R^g_{i'jk}{}^{h} = -\frac{1}{2} \partial_{i'}\partial_{h'} g_{jk}$,
\medbreak\qquad
$R^g_{i'jk}{}^{h'}  =
    -\frac{1}{2}\left(\partial_h\partial_{i'}g_{jk}-\partial_k\partial_{i'}g_{jh}\right)
    $\smallbreak\qquad\qquad$
      -\frac{1}{4} \left(\partial_{s'}g_{jk}\partial_{i'}g_{sh}+\partial_{s'}g_{jh}\partial_{i'}g_{sk}
        -2 \partial_{i'}(g_{hs}\partial_{s'}g_{jk}) \right)$,
\medbreak\qquad
$R^g_{ijk'}{}^{h'} = -\frac{1}{2}\left( \partial_{j}\partial_{k'}
            g_{ih}
            - \partial_{i}\partial_{k'} g_{jh}
            \right)
            -\frac{1}{4} \left(
         \partial_{k'} g_{is} \partial_{s'} g_{jh}
            - \partial_{k'} g_{js} \partial_{s'} g_{ih}
         \right)$,
\medbreak\qquad
$R^g_{i'jk'}{}^{h'} = \frac{1}{2} \partial_{i'}\partial_{k'} g_{jh}$.

\section{Modified Riemannian Extensions}\label{sect-4}
\subsection{The geometry of the cotangent bundle}
We refer to \cite{YI-73} for further details concerning the material of this Section.
Let $T^*M$ be the cotangent bundle of an $n$ dimensional manifold $M$
and let $\pi:T^*M\rightarrow M$ be the projection. Let $\tilde{p}=(p,\omega)$ where $p\in M$ and $\omega\in T_p^*M$ denote a point of $T^*M$.
Local coordinates $(x_i)$ in a neighborhood $U$ of $M$ induce
coordinates $(x_i,x_{i'})$ in $\pi^{-1}(U)$ where we decompose
$$\omega=\sum x_{i'}dx^i\,.$$

For each vector field $X$ on $M$, define a function $\iota
X:T^*M\rightarrow\mathbb{R}$ by
$$\iota X(p,\omega)=\omega(X_p)\,.$$
We may expand
$X=X^j\partial_j$ and express:
$$\iota X(x_i,x_{i'})=\sum x_{i'}X^i\,.$$
\begin{lemma}\label{lem-4.1}
 Let $\tilde Y,\tilde Z\in C^\infty(T(T^*M))$ be smooth vector fields on $T^*M$. Then $\tilde Y=\tilde Z$ if and only if
$\tilde Y(\iota X)=\tilde Z(\iota X)$ for all smooth vector fields $X\in C^\infty(TM)$.
\end{lemma}

Let $X\in C^\infty(TM)$ be a vector field on $M$. The {\it complete lift} $X^C$ is, by Lemma \ref{lem-4.1}, characterized by the identity
$$X^C(\iota Z)=\iota [X,Z]\quad\text{for all}\quad Z\in C^\infty(TM)\,.$$
We then have $T_{(p,\omega)}(T^*M)=\{X^C_{p,\omega}:X\in C^\infty(TM)\}$ and consequently
\begin{lemma}
A $(0,s)$-tensor field on $T^*M$ is characterized by its evaluation on complete lifts of vector fields on $M$.
\end{lemma}

Let $T$ be a tensor field of  type $(1,1)$ on $M$, i.e. $T\in C^\infty(\operatorname{End}(TM))$. We define a $1$-form $\iota(T)\in
C^\infty(T^*(T^*M))$ which is characterized by the identity
\begin{equation}\label{eqn-4.a}
\iota(T)(X^C)=\iota(TX)\,.
\end{equation}

\subsection{The Riemannian extension}
Let $\nabla$ be a torsion free affine connection on $M$. The {\it Riemannian extension} $g_\nabla$ is
the pseudo-Riemannian metric $g_\nabla$ on $N:=T^*M$ of neutral signature $(n,n)$ characterized by the identity:
$$g_\nabla(X^C,Y^C)=-\iota(\nabla_XY+\nabla_YX)\,.$$
Let $\nabla_{\partial_i}\partial_j=\Gamma^\nabla_{ij}{}^k\partial_k$ give the Christoffel symbols of the connection $\nabla$. Then:
$$g_\nabla= 2\, dx^i\circ dx^{i'}-2x_{k'}\Gamma^\nabla_{ij}{}^kdx^i\circ dx^j\,.$$

Riemannian extensions were originally defined by Patterson and Walker
\cite{PWalker} and further investigated in \cite{Af} thus relating
pseudo-Riemannian properties of $T^*M$ with the affine structure of
the base manifold $(M,\nabla)$. Moreover, Riemannian extensions were also
considered in \cite{GKVV} in relation to Osserman manifolds (see
also \cite{D}).

\subsection{The modified Riemannian extension}
Let $\Phi\in C^\infty(S^2(T^*M))$ be a symmetric $(0,2)$-tensor field on $M$ and let
$T,S\in C^\infty(\operatorname{End}(TM))$ be  tensor fields
of type $(1,1)$ on $M$.
The \emph{modified Riemannian extension} is the neutral signature metric on
$T^*M$ defined by
$$g_{\nabla,\Phi,T,S}:=\iota T\circ\iota S + g_\nabla+\pi^*\Phi\,.$$
In a system of local coordinates one has
\begin{equation}\label{eqn-4.b}
g_{\nabla,\Phi,T,S}= 2\, dx^i\circ dx^{i'}+\{  \mbox{$\frac12$} x_{r'} x_{s'} (T_i^r S_j^s + T_j^r S_i^s)
 + \Phi_{ij}(x) -2x_{k'}\Gamma^\nabla_{ij}{}^k\} dx^i\circ dx^j\,.
\end{equation}
The case when $T=c\operatorname{id}$ and $S=\operatorname{id}$ is important and plays a central role in our treatment. More precisely, if
$$g_{\nabla,\Phi,c}:=
c\cdot\iota\operatorname{id}\circ\iota\operatorname{id}+g_\nabla+\pi^*\Phi,$$
then one has in a system of local coordinates that
\begin{equation}\label{eqn-4.c}
g_{\nabla,\Phi,c}= 2\, dx^i\circ dx^{i'}+\{ c\, x_{i'} x_{j'}  + \Phi_{ij}(x) -2x_{k'}\Gamma^\nabla_{ij}{}^k\} dx^i\circ dx^j\,.
\end{equation}
These metrics are Walker metrics on $T^*M$ where the tensor $B_{ij}(x,x^\prime)$ of Equation (\ref{eqn-2.a})
is a quadratic function of $x^\prime$ (and affine if $c=0$). The parallel degenerate distribution $\mathcal{D}=\ker(\pi_*)$ and the
scalar curvature is a suitable multiple (depending on the dimension) of the parameter $c$.

The modified Riemannian extensions $g_{\nabla,\Phi,0}$ have been used in \cite{BCGHV} to construct Kaehler and para-Kaehler Osserman metrics with one-side bounded
(para) holomorphic sectional curvature.

\subsection{The proof of Theorem \ref{thm-1.2}}
Let $g=g_{\nabla,\Phi,c}=c\cdot\iota\operatorname{id}\circ\iota\operatorname{id}+g_\nabla+\pi^*\Phi$ and let $\tau^g$ be the scalar curvature. The trace free Ricci tensor $\rho_0^g=\rho^g-\frac{\tau^g}{2n}g$ can then  be determined to be
$$
   \rho_0^g = 2\, \pi^*\rho^{\nabla,s} -  \textstyle\frac12c(n-1)  \pi^*\Phi.
$$
Theorem \ref{thm-1.2} now follows.\hfill\qedbox

\section{Para-Kaehler manifolds}\label{sect-5}

A para-Kaehler  manifold is a symplectic manifold $N$ admitting two
transversal Lagrangian foliations
(see \cite{CFG96, IZ2005}). Such a structure induces a
decomposition of the tangent bundle $TN$ into the Whitney sum of
Lagrangian subbundles $L$ and $L'$, that is, $TN=L\oplus L'$. By
generalizing this definition, an almost para-Hermitian manifold is
defined to be an almost symplectic manifold $(N,\Omega)$ whose
tangent bundle splits into the Whitney sum of Lagrangian subbundles.
This definition implies that the $(1,1)$-tensor field $J$ defined by
$J=\pi_L-\pi_{L'}$ is an almost paracomplex structure, that is
$J^2=\operatorname{id}$ on $N$, such that $\Omega(JX,JY)$ $=$ $-\Omega(X,Y)$ for
all vector fields $X$, $Y$ on $N$, where $\pi_L$ and $\pi_L'$ are
the projections of $TN$ onto $L$ and $L'$, respectively. The 2-form
$\Omega$ induces a non-degenerate $(0,2)$-tensor field $g$ on $N$
defined by $g(X,Y)$ $=$ $\Omega(X,JY)$, where $X$, $Y$ are vector
fields on $N$. Now the relation between the almost paracomplex and
the almost symplectic structures on $N$ shows that $g$ defines a
pseudo-Riemannian metric of signature $(n,n)$ on $N$ and moreover, one has that
$g(JX,Y)+g(X,JY)=0$, where $X$, $Y$ are vector fields on $N$.  We refer to \cite{CFG96} for further details on paracomplex
geometry.

The special significance of the para-Kaehler condition is
equivalently stated  in terms of the parallelizability of the
paracomplex structure with respect to the Levi-Civita connection of
$g$, that is $\nabla^g J=0$.  The $\pm$ eigenspaces $\mathcal{D}_\pm$ of the paracomplex structure $J$ are null distributions. Moreover,
since
$J$ is parallel in the para-Kaehler setting, the distributions $\mathcal{D}_\pm$ are parallel. This shows that any para-Kaehler structure
$(g,J)$ has necessarily an underlying Walker metric.

\subsection{Proof of Theorem \ref{thm-1.1}}  Choose an affine manifold $(M,\nabla)$, with $\nabla$ a flat connection on $M$.
We normalize the choice of coordinates so the associated Christoffel symbols vanish. Consider the modified Riemannian extension
$$
g=g_{\nabla,c}=2\, dx^i\circ dx^{i'}+ c\, x_{i'}x_{j'}dx^i\circ dx^j\,.
$$
Let $\Omega=dx^i\wedge dx^{i'}$ be the symplectic form.
Now, the associated almost para-Hermitian structure $J$, defined by
$\Omega(X,Y)=g(JX,Y)$, is given by (sum over $j$)
\medbreak\qquad
   $ J\partial_i   =   \partial_i - c\, x_{i'} x_{j'}
    \partial_{j'}$,
    \qquad
    $J\partial_{i'}  =  - \partial_{i'}$,
\medbreak\noindent
for $i=1,\dots,n$, and a direct calculation shows that the Nijenhuis
tensor $$N_J(X,Y)=[JX,JY]-J[JX,Y]-J[X,JY]-[X,Y]=0\,.$$
As $J$
is integrable, and as $\Omega$ is closed,
$(T^*M,g,J)$ is a para-Kaehler manifold.

We finish the proof showing that $(T^*M,g,J)$ has constant
para holomorphic sectional curvature $c$. First, a straightforward
calculation from the expressions for the curvature given in Section \ref{Sect-TRC} shows that the
non-null components of the curvature tensor of the Levi-Civita connection are determined by (we do not sum over repeated indices)

\medbreak\noindent\begin{tabular}{llll}
    $R^g_{ij\alpha}{}^{i}=\frac{c^2}{4} x_{j'}  x_{\alpha'}$,&\hspace*{-.2cm}
    $R^g_{i'ii}{}^i=-c$,&\hspace*{-.2cm}
    $R^g_{i'ik}{}^k=\frac{-c}{2}$,  &\hspace*{-.2cm}
    $R^g_{i'ji}{}^j=\frac{-c}{2}$,
\\[0.1in]
    $R^g_{i'ii}{}^{i'}=c^2 x_{i'}^2$,&\hspace*{-.2cm}
    $R^g_{i'ik}{}^{k'}=\frac{c^2}{2} x_{k'}^2$,&\hspace*{-.2cm}
    $R^g_{i'ii}{}^{h'}=\frac{3c^2}{4} x_{i'} x_{h'}$,&\hspace*{-.2cm}
    $R^g_{i'ik}{}^{i'}=\frac{3c^2}{4} x_{i'} x_{k'}$,
\\[0.1in]
    $R^g_{i'ik}{}^{h'}=\frac{c^2}{2} x_{k'} x_{h'}$,&\hspace*{-.2cm}
    $R^g_{i'ji}{}^{i'}=\frac{c^2}{2} x_{i'} x_{j'}$,  &\hspace*{-.2cm}
    $R^g_{i'\alpha i}{}^{h'}=\frac{c^2}{4} x_{\alpha'} x_{h'}$, &\hspace*{-.2cm}
    $R^g_{i'\alpha k}{}^{i'}=\frac{c^2}{4} x_{\alpha'} x_{k'}$,
\\[0.1in]
    $R^g_{iji'}{}^{\alpha'}= \frac{-c^2}{4} x_{j'} x_{\alpha'}$,&\hspace*{-.2cm}
    $R^g_{i'ii'}{}^{i'}=c$,&\hspace*{-.2cm}
    $R^g_{i'ik'}{}^{k'}=\frac{c}{2}$, &\hspace*{-.2cm}
    $R^g_{i'jj'}{}^{i'}=\frac{c}{2}$,
\end{tabular}
\medbreak\noindent
adding the symmetry condition $R^g_{abc}{}^d=-R^g_{bac}{}^d$, and where $i$,
$j$, $k$, $h$ are different indexes in $\{1,\dots,n\}$. Now, for any
vector field $X=(\alpha_i \partial_i +
\alpha_{i'}\partial_{i'})$, a long but straightforward calculation
from the above relations shows that (sum over $i$ and $j$)
\medbreak\qquad
    $R^g(JX,X)X = -\varepsilon_X\,
     \left(
    \alpha_i  R^g_{i'ii}{}^i   \partial_i
    +2 \alpha_i   R^g_{i'ji}{}^{i'} \partial_{j'}
    -\alpha_i  R^g_{i'ii}{}^{i'} \partial_{i'}
    -\alpha_{i'} R^g_{i'ii}{}^i \partial_{i'}
    \right)$
\medbreak\qquad\qquad
     $ = c\, \varepsilon_X\, \left( \alpha_i\partial_i
        - \alpha_i c\, x_{i'}  x_{j'}\partial_{j'}
        - \alpha_{i'}\partial_{i'} \right)$
      $= c\, \varepsilon_X\, \left( \alpha_i J\partial_i
        + \alpha_{i'} J\partial_{i'} \right)$.
\medbreak\noindent
Thus, $R^g(JX,X)X=c \,g(X,X)\, JX$, showing that $(T^*M,g,J)$ has
constant para holomorphic sectional curvature $c$ and finishing the
proof.\hfill\qedbox

\section{Four Dimensional geometry}\label{sect-7}

\subsection{Self-dual Walker metrics}

In the particular case of $n=4$, we choose suitable coordinates
$(x_1,x_2,x_{1'},x_{2'})$ where the Walker metric takes the form:
\begin{equation}\label{eqn-7.a}
    g = 2\, dx^1\circ dx^{1'}+ 2\,  dx^2\circ dx^{2'}
     +\, a \, dx^{1}\circ dx^{1}  + b\,  dx^{2}\circ dx^{2} + 2\, c\, dx^{1}\circ dx^{2},
\end{equation}
for some functions $a$, $b$ and $c$ depending on the coordinates
$(x_1,x_2,x_{1'},x_{2'})$.

Considering the Riemann curvature tensor as an endomorphism of $\Lambda^2(N)$, we have the
following $O(2,2)$-decomposition
$$
R^g\equiv \frac{\tau^g}{12}\;\operatorname{id}_{\Lambda^2} +\rho^g_0 + W^g:\,\,\Lambda^2 \rightarrow \Lambda^2,
$$
where $W^g$ denotes the Weyl conformal curvature tensor, $\tau^g$ the scalar curvature, and $\rho_0^g$ the traceless Ricci tensor,
$$\rho_0^g(X,Y)=\rho^g(X,Y)-\frac{\tau^g}{4}\,g(X,Y)\,.$$
The Hodge star operator
$\star:\Lambda^2\rightarrow \Lambda^2$ associated with any {metric of signature $(2,2)$} induces a further
splitting $\Lambda^2 =\Lambda^2_+\oplus\Lambda^2_-$, where $\Lambda^2_\pm$ denotes the $\pm
1$-eigenspaces of the Hodge star operator, that is
$$\Lambda^2_\pm=\{\alpha\in\Lambda^2(N):
\star\alpha=\pm\alpha\}\,.$$
 Correspondingly, the curvature tensor decomposes as
$$
R^g\equiv
\frac{\tau^g}{12}\;\operatorname{id}_{\Lambda^2} + \rho^g_0 + W^g_++W^g_-\quad\text{for}\quad W^g_\pm=\textstyle\frac{1}{2}\;(W^g\pm\star W^g)\,.
$$
Recall that a pseudo-Riemannian $4$-manifold is called \emph{self-dual} (resp.,
\emph{anti-self-dual}) if $W^g_-=0$ (resp., $W^g_+=0$).

Self-dual Walker metrics have been previously investigated in \cite{DR-GR-VL-06} (see also
\cite{DM}), where a local description in Walker coordinates of such metrics was obtained.
As a further application, self-dual Walker metrics    can be completely described in terms
of the modified   Riemannian extension as follows.

\begin{theorem}\label{thm-7.2}
A four dimensional Walker metric   is self-dual if and only
if it is locally isometric to the cotangent bundle $T^*\Sigma$ of an affine
surface $(\Sigma,\nabla)$, with metric tensor
$$
g=\iota X(\iota \operatorname{id}\circ\iota \operatorname{id})+ \iota \operatorname{id}\circ\iota
T +g_{\nabla}+\pi^*\Phi\,
$$
where $X$, $T$, $\nabla$ and $\Phi$ are a vector field, a
$(1,1)$-tensor field, a torsion free affine connection and a
symmetric $(0,2)$-tensor field on $\Sigma$, respectively.
\end{theorem}

\begin{proof}
It follows from \cite{DR-GR-VL-06} that the metric of Equation (\ref{eqn-7.a}) is self-dual if and only if the
functions $a$, $b$, $c$ have the form
\begin{equation}\label{eqn-7.b}
\begin{array}{l}
   a(x_1,x_2,x_{1'},x_{2'})\! =\! x_{1'}^3 \mathcal{A}\! +\! x_{1'}^2 \mathcal{B} \!+\! x_{1'}^2 x_{2'} \mathcal{C}
   \!+\! x_{1'} x_{2'} \mathcal{D} \!+\! x_{1'} P \!+\! x_{2'} Q \!+\! \xi,
   \\[0.1in]
   b(x_1,x_2,x_{1'},x_{2'}) \!=\! x_{2'}^3 \mathcal{C} \!+\! x_{2'}^2 \mathcal{E} \!+ \!x_{1'} x_{2'}^2 \mathcal{A}
   \!+\! x_{1'} x_{2'} \mathcal{F} \!+\! x_{1'} S \!+\! x_{2'} T \!+\! \eta,
   \\[0.1in]
   c(x_1,x_2,x_{1'},x_{2'}) \!=\! \frac{1}{2} x_{1'}^2 \mathcal{F} \!+\! \frac{1}{2} x_{2'}^2 \mathcal{D}
   \!+\! x_{1'}^2 x_{2'} \mathcal{A} \!+\! x_{1'} x_{2'}^2 \mathcal{C} \!+\! \frac{1}{2} x_{1'} x_{2'} ( \mathcal{B} \!+\! \mathcal{E} )
   \\[0.1in]
   \phantom{c(x_1,x_2,x_{1'},x_{2'}) \!=\!}
   \!+\! x_{1'} U \!+\! x_{2'} V \!+\! \gamma,
\end{array}
\end{equation}
where all capital, calligraphic and Greek letters stand for arbitrary smooth functions depending
only on the coordinates $(x_1,x_2)$.

For a vector field $X=\mathcal{A}(x_1,x_2)\partial_1+\mathcal{C}(x_1,x_2)\partial_2$ on  $\Sigma$ we have
\[
    \iota X= x_{1'}\mathcal{A}(x_1,x_2)+x_{2'}\mathcal{C}(x_1,x_2),
\]
and hence
\[
\begin{array}{l}
   (\iota X\cdot\iota \operatorname{id}\circ\iota \operatorname{id})_{11} =
   x_{1'}^3 \mathcal{A}(x_1,x_2)+x_{1'}^2x_{2'} \mathcal{C}(x_1,x_2),
   \\[0.05in]
   (\iota X\cdot\iota \operatorname{id}\circ\iota \operatorname{id})_{12} =
   x_{1'}^2x_{2'} \mathcal{A}(x_1,x_2)+x_{1'}x_{2'}^2 \mathcal{C}(x_1,x_2),
   \\[0.05in]
   (\iota X\cdot\iota \operatorname{id}\circ\iota \operatorname{id})_{22} =
   x_{1'}x_{2'}^2 \mathcal{A}(x_1,x_2)+x_{2'}^3 \mathcal{C}(x_1,x_2).
\end{array}
\]
Next, let $T$ be a $(1,1)$-tensor field on $\Sigma$ with components
\[
T^1_1=\mathcal{B}(x_1,x_2), \quad T^2_1=\mathcal{D}(x_1,x_2), \quad T^1_2=\mathcal{F}(x_1,x_2),
\quad T^2_2=\mathcal{E}(x_1,x_2).
\]
It follows from the definition of  $\iota T$ in
Equation (\ref{eqn-4.a}) that:
\[
\begin{array}{l}
   (\iota T)_1 = x_{1'}\mathcal{B}(x_1,x_2)+x_{2'}\mathcal{D}(x_1,x_2),
   \\[0.05in]
   (\iota T)_2 = x_{1'}\mathcal{F}(x_1,x_2)+x_{2'}\mathcal{E}(x_1,x_2)\\
\end{array}
\]
and therefore
\[
\begin{array}{l}
    (\iota T\circ\iota\operatorname{id})_{11} =  x_{1'}^2 \mathcal{B}(x_1,x_2)
    +x_{1'}x_{2'} \mathcal{D}(x_1,x_2),
    \\[0.05in]
    (\iota T\circ\iota\operatorname{id})_{12} =
    \frac{1}{2}\left(x_{1'}^2 \mathcal{F}(x_1,x_2)
    +x_{2'}^2 \mathcal{D}(x_1,x_2)\right.
    \\[0.05in]
    \phantom{(\iota T\circ\iota\operatorname{id})_{12}=\frac{1}{2}(}
    \left. +x_{1'}x_{2'}(\mathcal{B}(x_1,x_2)+\mathcal{E}(x_1,x_2))\right),
    \\[0.05in]
    (\iota T\circ\iota\operatorname{id})_{22} =  x_{1'}x_{2'}\mathcal{F}(x_1,x_2)
    +x_{2'}^2 \mathcal{E}(x_1,x_2).
\end{array}
\]
Now the result follows from Equation (\ref{eqn-7.b}).
\end{proof}

\medskip

Note that, as a direct consequence of  Equation (\ref{eqn-7.b}),  any modified Riemannian
extension $g_{\nabla,\Phi,0}$ is necessarily a self-dual Walker metric.  Recently, the authors showed that Ivanov--Petrova self-dual Walker $4$-metrics correspond to modified
Riemannian extensions $g_{\nabla,\Phi,0}$ \cite{C-GR-VL}, and the same was shown by Derdzinski for Ricci-flat  self-dual  Walker $4$-metrics   \cite{D}.  For sake of
completeness, we recall the following \cite[Thm. 6.1]{D}, which strengthen the main result in \cite{GKVV}.

\begin{theorem}\label{thm-7.1}
A four dimensional Ricci flat self-dual Walker metric is locally
isometric to the cotangent bundle $T^*\Sigma$ of an affine surface
$(\Sigma,\nabla)$ equipped with the modified Riemannian extension
$g_{\nabla,\Phi,0}=g_\nabla +\pi^*\Phi$, where $\nabla$ is a torsion free connection
with skew-symmetric Ricci tensor which expresses in adapted
coordinates $(x_1,x_2)$ by
\[
\Gamma^\nabla_{11}{}^1=-\partial_1\varphi, \qquad \Gamma^\nabla_{22}{}^2=\partial_2\varphi
\]
for an arbitrary function $\varphi$, and where $\Phi$ is an
arbitrary symmetric $(0,2)$-tensor field on $\Sigma$.
\end{theorem}

\medskip

\subsection{Osserman $4$-metrics with non-diagonalizable Jacobi operators}

Let $\mathcal{N}$ be
a pseudo-Riemannian manifold of signature $(2,2)$. Then, for each non-null vector
$X$, the induced metric on $X^\perp$ is of Lorentzian signature and thus the Jacobi operator
${\mathcal{J}^g}(X)=R^g(\,\cdot\,,X)X$, viewed as an endomorphism of $X^\perp$, corresponds to one of the following
possibilities
\cite{BBR01}:

\[
\begin{array}{cccc}
\left(\begin{array}{ccc}
\alpha & &\\
       &\beta&\\
       & &\gamma
\end{array}\right),\!\quad
& \left(\begin{array}{ccc}
\alpha &-\beta&\\
\beta&\alpha&\\
       & &\gamma
\end{array}\right),\!\quad
& \left(\begin{array}{ccc}
\alpha & &\\
       &\beta&\\
       & 1&\beta
\end{array}\right),\!\quad
& \left(\begin{array}{ccc}
\alpha & &\\
    1   &\alpha&\\
       & 1&\alpha
\end{array}\right).
\\
\noalign{\medskip} \!\!\!\!\emph{Type Ia}\,\, & \!\!\!\!\emph{Type Ib}\,\,\, &
\!\!\emph{Type II}\,\, & \emph{Type III}
\end{array}
\]
Type Ia Osserman metrics correspond to real, complex and
paracomplex space forms, Type Ib Osserman metrics do not exist
\cite{BBR01}, and Types II and III Osserman metrics with non-nilpotent
 Jacobi operators have recently been  classified in
\cite{DR-GR-VL-06} and \cite{De}, respectively. Further, note
that any Type II Osserman metric whose Jacobi operators have
non-zero eigenvalues is necessarily a Walker metric.
Moreover, since four dimensional Osserman metrics are locally Einstein and self-dual, they can be
described by means of   modified Riemannian extensions as follows.

\begin{theorem}\label{thm-7.3}
A four dimensional Type II Osserman metric whose Jacobi operator  has non-zero eigenvalues is locally
isometric to the cotangent bundle $T^*\Sigma$ of an affine
surface $(\Sigma,\nabla)$, with metric tensor
$$
g=\frac{\tau}{6}\cdot \iota \operatorname{id}\circ\iota \operatorname{id} + g_{\nabla}+\frac{24}{\tau}\pi^*\Phi,
$$
where $\tau\neq 0$ denotes the scalar curvature of $(T^*\Sigma,g)$,
$\nabla$ is an arbitrary non-flat connection on $\Sigma$ and $\Phi$ is the symmetric
part of the Ricci tensor  of $\nabla$.
\end{theorem}

\begin{proof}
It follows after some straightforward calculations as in Theorem \ref{thm-7.2} that
any Type II Osserman metric of non-zero scalar curvature $\tau$ is obtained by the above modified
Riemannian extension  of a torsion free connection $\nabla$ given by
\[
\begin{array}{lll}
\Gamma^\nabla_{11}{}^1=-\frac{1}{2}P(x_1,x_2), &\quad
\Gamma^\nabla_{11}{}^2=-\frac{1}{2}Q(x_1,x_2), &\quad
\Gamma^\nabla_{12}{}^1=-\frac{1}{2}U(x_1,x_2), \\
\noalign{\medskip}
\Gamma^\nabla_{12}{}^2=-\frac{1}{2}V(x_1,x_2), &\quad
\Gamma^\nabla_{22}{}^1=-\frac{1}{2}S(x_1,x_2), &\quad
\Gamma^\nabla_{22}{}^2=-\frac{1}{2}T(x_1,x_2),
\end{array}
\]
just considering \cite[Thm. 3.1]{DR-GR-VL-06}.
\end{proof}

\begin{remark}\rm
The first examples of non Ricci flat Type II Osserman metrics were given in \cite{DGV1} as follows.
Let $N:=\mathbb{R}^4$ with usual coordinates $(x_1,x_2,x_3,x_4)$ and define a metric by
\begin{equation}\label{eqn-7.c}
\begin{array}{l}
g= 2(dx^1\circ dx^3+ dx^2\circ dx^4)
+(4kx_1^2-\frac{1}{4k}f(x_4)^2)dx^3\circ dx^3\\
\quad+ 4kx_2^2 dx^4\circ dx^4
+2(4kx_1x_2+x_2f(x_4)-\frac{1}{4k}f'(x_4)) dx^3\circ dx^4,\vphantom{\vrule height 11pt}
\end{array}
\end{equation}
where $k$ is a non-zero constant and $f(x_4)$ is an arbitrary function.

Now, an easy calculation shows that  Equation (\ref{eqn-7.c}) is nothing but the modified  Riemannian extension
$g=4k\cdot \iota \operatorname{id}\circ\iota \operatorname{id} + g_{\nabla}+\frac{1}{k}\pi^*\Phi$
of the torsion free connection $\nabla$ given by
$\Gamma^\nabla_{12}{}^2=-\frac{1}{2}f(x_2)$, whose Ricci tensor is given by
$$\rho^\nabla=-\textstyle\frac{1}{4}f(x_2)^2dx^1\otimes dx^1 -\frac12f'(x_2)dx^1\otimes dx^2\,.
$$
This is neither symmetric nor skew-symmetric. We symmetrize to see
$$
\Phi=\textstyle-\frac{1}{4}f(x_2)^2dx^1\circ dx^1 -\frac{1}{2}f'(x_2)dx^1\circ dx^2\,.
$$
\end{remark}

\begin{remark}\rm
Four dimensional Type II Osserman metrics have been studied intensively during the last years. The existence of
many nilpotent examples suggested that the family of Osserman metrics with two-step nilpotent Jacobi operators
was larger than the non-nilpotent one. However, this seems not to be true
(see Theorems \ref{thm-7.1} and \ref{thm-7.3}).
\end{remark}

\section{Six Dimensional geometry}\label{sect-6}
Let $p=3$. We consider on  $\mathbb{R}^3$ the torsion free
connection $\nabla$ with the only non-zero Christoffel symbol given
by $\nabla_{\partial_1}\partial_1 = x_2 \partial_3$. This connection
is Ricci flat, but not flat. We set $\Phi=0$ and take $c=1$ in
Equation (\ref{eqn-4.c}) to define a metric $g=g_{\nabla,1}$ where
the tensor $B$ of Equation (\ref{eqn-2.a}) is given by:
\[
    B=
    \left(
    \begin{array}{ccc}
             x_{1'}^2 - 2 x_2 x_{3'} &  x_{1'} x_{2'} &
             x_{1'} x_{3'}
            \\[0.035in]
             x_{1'} x_{2'} &  x_{2'}^2   &
             x_{2'} x_{3'}
            \\[0.035in]
             x_{1'} x_{3'} &  x_{2'} x_{3'} &
             x_{3'}^2
    \end{array}
    \right).
\]

\subsection{Proof of Theorem \ref{thm-1.3}}

The curvature tensor of the Levi-Civita connection is given
by:
$$
\begin{array}{l}
    R^g_{1212}=\frac{1}{2} x_{2'} x_{3'} (x_2 x_{2'}-4),\qquad
    R^g_{1213}=\frac{1}{2} x_{3'}^2 (x_2 x_{2'} - 2),
    \\[0.075in]
    R^g_{1211'} = - R^g_{1222'} = -R^g_{1323'} = -R^g_{13'23} =\frac{1}{4} x_{1'} x_{2'},
    \\[0.075in]
    R^g_{1212'}=R^g_{1313'}=-\frac{1}{4}(x_{1'}^2-2x_2 x_{3'}),\qquad
    R^g_{1221'} = - R^g_{2323'} = \frac{1}{4} x_{2'}^2,
    \\[0.075in]
    R^g_{1213'}=-R^g_{11'11'}= -R^g_{22'22'}=-R^g_{33'33'} = -1,
    \\[0.075in]
    R^g_{1231'} = R^g_{1321'}= R^g_{2322'} = -R^g_{2333'} = \frac{1}{4} x_{2'} x_{3'},
    \\[0.075in]
    R^g_{1232'}= -R^g_{1311'} = R^g_{1333'} = -R^g_{12'23} =-\frac{1}{4} x_{1'} x_{3'},
    \\[0.075in]
    R^g_{1313} = \frac{1}{2} x_2 x_{3'}^3, \qquad
    R^g_{1331'} = R^g_{2332'} = \frac{1}{4} x_{3'}^2,
    \\[0.075in]
    R^g_{11'22'} = R^g_{11'33'} = R^g_{12'21'}=R^g_{13'31'}=R^g_{22'33'}=R^g_{23'32'} = \frac{1}{2}.
\end{array}
$$

The eigenvalues of the Jacobi operator of an Osserman metric change
sign when passing from timelike to spacelike directions. Thus, for
the purpose of studying the Osserman property, it is convenient to
consider the operator given by setting
$\widetilde{\mathcal{J}^g}(v)=g(v,v)^{-1}{\mathcal{J}^g}(v)$
associated to each non-null vector $v$, whose eigenvalues must be
constant if and only if the manifold is Osserman.

We now determine the Jacobi operator.
Let $v=\sum_{i=1}^3(\alpha_i\partial_i+\alpha_{i'}\partial_{i'})$ be a
non-null vector, where
$\{\partial_i,\partial_{i'}\}$ denotes the coordinate basis. The associated Jacobi operator
${\mathcal{J}^g}(v)=R^g(\,\cdot\,,v)v$
can be expressed, with respect to the coordinate basis, as
$$
{\mathcal{J}^g}(v)= \frac{1}{4} \left(
\begin{array}{cccccc}
a_{11} & a_{12} & a_{13} & -4\alpha_1^2 & -4\alpha_1\alpha_2 & -4\alpha_1 \alpha_3
\\
a_{21} & a_{22} & a_{23} & -4\alpha_1\alpha_2 & -4\alpha_2^2 & -4\alpha_2 \alpha_3
\\
a_{31} & a_{32} & a_{33} & -4\alpha_1 \alpha_3 & -4\alpha_2\alpha_3 & -4\alpha_3^2
\\
a_{1'1} & a_{1'2} & a_{1'3} & a_{1'1'} & a_{1'2'} & a_{1'3'}
\\
a_{2'1} & a_{2'2} & a_{2'3} & a_{2'1'} & a_{2'2'} & a_{2'3'}
\\
a_{3'1} & a_{3'2} & a_{3'3} & a_{3'1'} & a_{3'2'} & a_{3'3'}
\end{array}\right)
$$
with

\begin{flushleft}
$
\begin{array}{l}
    a_{11}= x_{2'}^2 \alpha_2^2 + 2 x_{2'} x_{3'} \alpha_2 \alpha_3 + x_{3'}^2 \alpha_3^2 +
    x_{1'} (x_{2'} \alpha_1 \alpha_2 + x_{3'} \alpha_1 \alpha_3)
    \\[0.035in]
    \phantom{a_{11}=}
    +  2 (2 \alpha_1 \alpha_{1'} + \alpha_2 \alpha_{2'} + \alpha_3 \alpha_{3'}),
\end{array}
$
\end{flushleft}
\begin{flushleft}
$
\begin{array}{l}
    a_{1j}= -x_{1'} x_{j'} \alpha_1^2 - x_{2'} x_{j'} \alpha_1 \alpha_2 - x_{j'} x_{3'} \alpha_1 \alpha_3 + 2 \alpha_1 \alpha_{j'},\,\, j=2,3,
\end{array}
$
\end{flushleft}
\begin{flushleft}
$
\begin{array}{l}
    a_{21}= -x_{1'}^2 \alpha_1 \alpha_2 + 2 x_2 x_{3'} \alpha_1 \alpha_2
    -  x_{1'} (x_{2'} \alpha_2^2 + x_{3'} \alpha_2 \alpha_3) + 2 \alpha_2 \alpha_{1'},
\end{array}
$
\end{flushleft}
\begin{flushleft}
$
\begin{array}{l}
    a_{22}= \varepsilon_v - x_{1'} x_{2'} \alpha_1 \alpha_2 - x_{2'}^2 \alpha_2^2 - x_{2'} x_{3'} \alpha_2 \alpha_3 +
 2 \alpha_2 \alpha_{2'},
\end{array}
$
\end{flushleft}
\begin{flushleft}
$
\begin{array}{l}
    a_{23}= -x_{1'} x_{3'} \alpha_1 \alpha_2 - x_{2'} x_{3'} \alpha_2^2 - x_{3'}^2 \alpha_2 \alpha_3 + 2 \alpha_2 \alpha_{3'},
\end{array}
$
\end{flushleft}
\begin{flushleft}
$
\begin{array}{l}
    a_{31}= -4 \alpha_1 \alpha_2 - x_{1'}^2 \alpha_1 \alpha_3 + 2 x_2 x_{3'} \alpha_1 \alpha_3
    -  x_{1'} (x_{2'} \alpha_2 \alpha_3 + x_{3'} \alpha_3^2) + 2 \alpha_3 \alpha_{1'},
\end{array}
$
\end{flushleft}
\begin{flushleft}
$
\begin{array}{l}
    a_{32}= 4 \alpha_1^2 - x_{1'} x_{2'} \alpha_1 \alpha_3 - x_{2'}^2 \alpha_2 \alpha_3 - x_{2'} x_{3'} \alpha_3^2 +
 2 \alpha_3 \alpha_{2'},
\end{array}
$
\end{flushleft}
\begin{flushleft}
$
\begin{array}{l}
    a_{33}= \varepsilon_v  - x_{1'} x_{3'} \alpha_1 \alpha_3 - x_{2'} x_{3'} \alpha_2 \alpha_3 - x_{3'}^2 \alpha_3^2 +
 2 \alpha_3 \alpha_{3'},
\end{array}
$
\end{flushleft}
\begin{flushleft}
$
\begin{array}{l}
    a_{1'1}= 8 x_{2'} x_{3'} \alpha_2^2 + 8 x_{3'}^2 \alpha_2 \alpha_3 - 4 x_{1'}^2 \alpha_1 \alpha_{1'} +
 8 x_2 x_{3'} \alpha_1 \alpha_{1'} - 4 \alpha_{1'}^2
 \\[0.035in]
    \phantom{a_{1'1}=}
    -  x_{1'} (4 x_{2'} \alpha_2 \alpha_{1'} - x_{3'} (4 \alpha_1 \alpha_2 - 4 \alpha_3 \alpha_{1'})) +
 8 \alpha_2 \alpha_{3'},
\end{array}
$
\end{flushleft}
\begin{flushleft}
$
\begin{array}{l}
    a_{1'2}= -4 x_{3'}^2 \alpha_1 \alpha_3 - x_{2'}^2 \alpha_2 \alpha_{1'} -
 x_{2'} x_{3'} (8 \alpha_1 \alpha_2 + \alpha_3 \alpha_{1'}) - 3 x_{1'}^2 \alpha_1 \alpha_{2'}
 \\[0.035in]
    \phantom{a_{1'2}=}
    +  6 x_2 x_{3'} \alpha_1 \alpha_{2'} -
 x_{1'} (x_{2'} (\alpha_1 \alpha_{1'} + 3 \alpha_2 \alpha_{2'}) +
    x_{3'} (4 \alpha_1^2 + 3 \alpha_3 \alpha_{2'}))
    \\[0.035in]
    \phantom{a_{1'2}=}
    - 4 (\alpha_{1'} \alpha_{2'} + \alpha_1 \alpha_{3'}),
\end{array}
$
\end{flushleft}
\begin{flushleft}
$
\begin{array}{l}
    a_{1'3}= -x_{2'} x_{3'} \alpha_2 \alpha_{1'} - x_{3'}^2 (4 \alpha_1 \alpha_2 + \alpha_3 \alpha_{1'}) -
 3 x_{1'}^2 \alpha_1 \alpha_{3'} + 6 x_2 x_{3'} \alpha_1 \alpha_{3'}
 \\[0.035in]
    \phantom{a_{1'3}=}
    - 4 \alpha_{1'} \alpha_{3'} -
 x_{1'} (3 x_{2'} \alpha_2 \alpha_{3'} + x_{3'} (\alpha_1 \alpha_{1'} + 3 \alpha_3 \alpha_{3'})),
\end{array}
$
\end{flushleft}
\begin{flushleft}
$
\begin{array}{l}
    a_{1'1'}= 4 \varepsilon_v - 3 x_{2'}^2 \alpha_2^2 - 6 x_{2'} x_{3'} \alpha_2 \alpha_3 -
 3 x_{3'}^2 \alpha_3^2 - x_{1'} (3 x_{2'} \alpha_1 \alpha_2 + 3 x_{3'} \alpha_1 \alpha_3)
 \\[0.035in]
    \phantom{a_{1'1'}=}
    -  4 \alpha_1 \alpha_{1'} - 6 \alpha_2 \alpha_{2'} - 6 \alpha_3 \alpha_{3'},
\end{array}
$
\end{flushleft}
\begin{flushleft}
$
\begin{array}{l}
    a_{1'2'}= 3 x_{1'}^2 \alpha_1 \alpha_2 - 6 x_2 x_{3'} \alpha_1 \alpha_2 +
 x_{1'} (3 x_{2'} \alpha_2^2 + 3 x_{3'} \alpha_2 \alpha_3) + 2 \alpha_2 \alpha_{1'},
\end{array}
$
\end{flushleft}
\begin{flushleft}
$
\begin{array}{l}
    a_{1'3'}= -4 \alpha_1 \alpha_2 + 3 x_{1'}^2 \alpha_1 \alpha_3 - 6 x_2 x_{3'} \alpha_1 \alpha_3 +
 x_{1'} (3 x_{2'} \alpha_2 \alpha_3 + 3 x_{3'} \alpha_3^2) + 2 \alpha_3 \alpha_{1'},
\end{array}
$
\end{flushleft}
\begin{flushleft}
$
\begin{array}{l}
    a_{2'1}= -4 x_{3'}^2 \alpha_1 \alpha_3 - 3 x_{2'}^2 \alpha_2 \alpha_{1'} -
 x_{2'} x_{3'} (4 \alpha_1 \alpha_2 + 3 \alpha_3 \alpha_{1'})
 +  2 x_2 x_{3'} \alpha_1 \alpha_{2'}
 \\[0.035in]
    \phantom{a_{2'1}=}
     - x_{1'}^2 \alpha_1 \alpha_{2'}
     -  x_{1'} (x_{3'} \alpha_3 \alpha_{2'} + x_{2'} (3 \alpha_1 \alpha_{1'} + \alpha_2 \alpha_{2'}))
      -  4 (\alpha_{1'} \alpha_{2'} + \alpha_1 \alpha_{3'}),
\end{array}
$
\end{flushleft}
\begin{flushleft}
$
\begin{array}{l}
    a_{2'2}= -4 x_{1'} x_{2'} \alpha_1 \alpha_{2'} - 4 x_{2'}^2 \alpha_2 \alpha_{2'} - 4 \alpha_{2'}^2 +
 4 x_{2'} x_{3'} (\alpha_1^2 - \alpha_3 \alpha_{2'}),
\end{array}
$
\end{flushleft}
\begin{flushleft}
$
\begin{array}{l}
    a_{2'3}= x_{3'}^2 (4 \alpha_1^2 - \alpha_3 \alpha_{2'}) - 3 x_{2'}^2 \alpha_2 \alpha_{3'} - 4 \alpha_{2'} \alpha_{3'} -
 x_{1'} (x_{3'} \alpha_1 \alpha_{2'} + 3 x_{2'} \alpha_1 \alpha_{3'})
 \\[0.035in]
    \phantom{a_{2'3}=}
    -  x_{2'} x_{3'} (\alpha_2 \alpha_{2'} + 3 \alpha_3 \alpha_{3'}),
\end{array}
$
\end{flushleft}
\begin{flushleft}
$
\begin{array}{l}
    a_{2'1'}= 3 x_{1'} x_{2'} \alpha_1^2 + 3 x_{2'}^2 \alpha_1 \alpha_2 + 3 x_{2'} x_{3'} \alpha_1 \alpha_3 +
 2 \alpha_1 \alpha_{2'},
\end{array}
$
\end{flushleft}
\begin{flushleft}
$
\begin{array}{l}
    a_{2'2'}= \varepsilon_v + 3 x_{1'} x_{2'} \alpha_1 \alpha_2 + 3 x_{2'}^2 \alpha_2^2 +
 3 x_{2'} x_{3'} \alpha_2 \alpha_3 + 2 \alpha_2 \alpha_{2'},
\end{array}
$
\end{flushleft}
\begin{flushleft}
$
\begin{array}{l}
    a_{2'3'}= 4 \alpha_1^2 + 3 x_{1'} x_{2'} \alpha_1 \alpha_3 + 3 x_{2'}^2 \alpha_2 \alpha_3 +
 3 x_{2'} x_{3'} \alpha_3^2 + 2 \alpha_3 \alpha_{2'},
\end{array}
$
\end{flushleft}
\begin{flushleft}
$
\begin{array}{l}
    a_{3'1}= -3 x_{2'} x_{3'} \alpha_2 \alpha_{1'} - 3 x_{3'}^2 \alpha_3 \alpha_{1'} - x_{1'}^2 \alpha_1 \alpha_{3'} +
 2 x_2 x_{3'} \alpha_1 \alpha_{3'} - 4 \alpha_{1'} \alpha_{3'}
 \\[0.035in]
    \phantom{a_{3'1}=}
    -  x_{1'} (x_{2'} \alpha_2 \alpha_{3'} + x_{3'} (3 \alpha_1 \alpha_{1'} + \alpha_3 \alpha_{3'})),
\end{array}
$
\end{flushleft}
\begin{flushleft}
$
\begin{array}{l}
    a_{3'2}= -3 x_{3'}^2 \alpha_3 \alpha_{2'} - x_{2'}^2 \alpha_2 \alpha_{3'} - 4 \alpha_{2'} \alpha_{3'} -
 x_{1'} (3 x_{3'} \alpha_1 \alpha_{2'} + x_{2'} \alpha_1 \alpha_{3'})
 \\[0.035in]
    \phantom{a_{3'2}=}
    -  x_{2'} x_{3'} (3 \alpha_2 \alpha_{2'} + \alpha_3 \alpha_{3'}),
\end{array}
$
\end{flushleft}
\begin{flushleft}
$
\begin{array}{l}
    a_{3'3}= -4 x_{1'} x_{3'} \alpha_1 \alpha_{3'} - 4 x_{2'} x_{3'} \alpha_2 \alpha_{3'} - 4 x_{3'}^2 \alpha_3 \alpha_{3'} -
 4 \alpha_{3'}^2,
\end{array}
$
\end{flushleft}
\begin{flushleft}
$
\begin{array}{l}
    a_{3'j'}= 3 x_{1'} x_{3'} \alpha_1 \alpha_j + 3 x_{2'} x_{3'} \alpha_2 \alpha_j
    + 3 x_{3'}^2  \alpha_3 \alpha_j + 2  \alpha_{3'} \alpha_j,\,\,j=1,2,
\end{array}
$
\end{flushleft}
\begin{flushleft}
$
\begin{array}{l}
    a_{3'3'}= \varepsilon_v + 3 x_{1'} x_{3'} \alpha_1 \alpha_3 + 3 x_{2'} x_{3'} \alpha_2 \alpha_3 +
 3 x_{3'}^2 \alpha_3^2 + 2 \alpha_3 \alpha_{3'}.
\end{array}
$
\end{flushleft}

The characteristic polynomial of the Jacobi operator is now seen to
be:
$$\textstyle p_\lambda(\widetilde{\mathcal{J}^g}(v))=\lambda(\lambda-1)(\lambda-
\frac{1}{4})^4,$$ and therefore $g$ is Osserman with eigenvalues
$\left\{ 0, 1,
\frac{1}{4},\frac{1}{4},\frac{1}{4},\frac{1}{4}\right\}$. Now, for
any unit vector $v$ set
\[\textstyle
    \mathcal{A}(v) = {\mathcal{J}^g}(v)\cdot
    ({\mathcal{J}^g}(v)- \varepsilon_v \operatorname{id})\cdot
    \left({\mathcal{J}^g}(v)-\frac{\varepsilon_v}{4}\operatorname{id}\right).
\]
For the particular choice of the unit vectors
\[\textstyle
    v=\partial_3 + \frac{1}{2}(\varepsilon -x_{3'}^2)\partial_{3'},\qquad
    \bar{v}=\partial_1-\frac{1}{2} (x_{1'}^2 - 2 x_2 x_{3'} -\bar{\varepsilon})\partial_{1'},
\]
we have $g(v,v)=\varepsilon$ and
$g(\bar{v},\bar{v})=\bar\varepsilon$, and a straightforward
calculation shows that $\mathcal{A}(v)=0$, while $(\mathcal{A}(\bar
v))_{32} = -\frac{3}{16}$. Therefore, at any point,
${\mathcal{J}^g}(v)$ diagonalizes while ${\mathcal{J}^g}(\bar{v})$
is not diagonalizable, and hence the metric is neither spacelike
Jordan Osserman nor timelike Jordan Osserman at any point.
\hfill\qedbox

\begin{remark}\rm
We make the following observations concerning the metric of Theorem
\ref{thm-1.3}.
\begin{enumerate}
\item The Jacobi operator ${\mathcal{J}^g}(v)$ associated to a unit vector $v$ is either diagonalizable, has a single $2\times 2$
Jordan block, or has a $3\times 3$ Jordan block, or has two $2\times
2$ Jordan blocks depending on the point and the vector $v$
considered; all possibilities can arise.
\gronkJORDONBLOCK
\gronkJNF
\gronkJORDONREMARK

\item Observe that the eigenvalues of the Jacobi operator of a
pseudo-Riemannian Osserman metric change sign from spacelike to
timelike vectors, and thus they are all zero for null vectors (cf.
\cite{GKV, G1-01}), which shows that any Osserman metric is null
Osserman. Hence, the metric of Theorem \ref{thm-1.3} is null Osserman; moreover, proceeding
as above, one checks that the null Jacobi operators can be two-step
nilpotent, or three-step nilpotent, or four-step nilpotent, changing
even at a fixed point, and therefore the metric is not pointwise
null Jordan Osserman.
\gronkRMKMOREJO

\item The metric of Theorem \ref{thm-1.3} is timelike and spacelike nilpotent Szab\'o, it is not symmetric, and it is not Jordan Szab\'o.
\gronkSZABO

\item A straightforward calculation shows that metric of Theorem \ref{thm-1.3} is not Ivanov--Petrova.
\end{enumerate}\end{remark}

\end{document}